\documentclass[12pt]{amsart}
\usepackage{amssymb,amscd, amsfonts, amsxtra}
\usepackage{amsmath, amsthm}
\usepackage{graphicx}
\usepackage{psfrag}
\usepackage{comment}
\usepackage{amsrefs}

\usepackage{amsmath, amsthm, tikz}
\usepackage{times}
\usepackage{anysize}

\usepackage{amsrefs}

\marginsize{1in}{1in}{1in}{1.5in}
\begin{document}

\newtheorem{lemma}{Lemma}[section]
\newtheorem{corollary}{Corollary}[section]
\newtheorem{proposition}{Proposition}
\newtheorem{theorem}{Theorem}[section]
\newtheorem{statement}{Statement}[section]
\newtheorem{definition}{Definition}[section]
\newtheorem{note}{Note}[section]
\newtheorem{claim}{Claim}[section]

\newcommand {\C}{{\mathbb C}}
\newcommand {\D}{{\mathbb D}}
\newcommand{\De}{{\Delta}}
\newcommand{\de}{{\delta}}
\newcommand {\ra} {\rightarrow}
\newcommand{\la}{{\lambda}}
\newcommand{\La}{{\Lambda}}
\newcommand{\si}{{\sigma}}
\newcommand{\Si}{{\Sigma}}
\newcommand{\Om}{{\Omega}}
\newcommand{\om}{{\omega}}
\newcommand{\cal}{\mathcal}

\newcommand{\tl}{\tilde}
\newcommand{\sm}{\setminus}
\newcommand{\inter}{\operatorname{int}}
\newcommand{\di}{\partial}

\newcommand{\AAA}{{\cal{A}}}
\newcommand{\FF}{{\cal {F}}}
\newcommand{\MM}{{\cal{M}}}
\newcommand{\TT}{{\cal{T}}}
\newcommand{\QC}{{\cal {QC}}}
\newcommand{\SSS}{{\cal{S}}}

\newcommand{\N}{{\Bbb{N}}}

\newcommand{\Ups}{{\Upsilon}}

\renewcommand{\Im}{\operatorname{Im}}
\newcommand{\diam}{\operatorname{diam}}

\title{Geometric proof of the $\lambda$-Lemma}
\maketitle

\begin{center}
Eric Bedford\footnote{bedford@indiana.edu, IMS, Stony Brook University, Stony Brook NY 11794-3660, USA} and  Tanya Firsova\footnote{tanyaf@math.ksu.edu, 138 Cardwell Hall, Manhattan, KS 66506, USA}
\end{center}

\begin{abstract} We give a geometric approach to the proof of the $\lambda$-lemma. In particular, we point out the role pseudoconvexity plays in the proof.

Nous donnons une approche g\'{e}om\'{e}trique de 
la preuve de la $\lambda$-lemma. En particulier, nous soulignons le r\'{o}le
pseudoconvexity joue dans la preuve.
\end{abstract}

\section{Introduction}

A holomorphic motion in dimension one is a family of injections $f_{\lambda}:A\to \hat{\mathbb C}$ over a complex manifold $\Lambda\ni\lambda$.  Holomorphic motions first appeared in \cite{MSS,L} where they were used to show that a generic rational map $f:\hat{\mathbb C}\to\hat{\mathbb C}$ is structurally stable.  This notion has since found numerous applications in holomorphic dynamics and  Teichm\"{u}ller Theory. 
Its usefulness comes from the fact that analyticity alone forces strong extendibility and regularity properties that are referred to as the $\lambda$-lemma. Let $\Delta$ be the unit disk in $\mathbb C$. 

\begin{theorem}$ $
\begin{itemize}
\item {\bf Extension $\lambda$-lemma} \cite{L}, \cite{MSS} Any holomorphic motion $f:\Delta\times A \to \hat\C$ extends to a holomorphic motion  $\Delta\times \bar A \ra \hat \C$. 

\item {\bf QC $\lambda$-lemma} \cite{MSS} The map $f(\lambda,a)$ is uniformly quasisymmetric in $a$.
\end{itemize}
\end{theorem}

Note that when $A$ has interior, $f(\lambda, a)$ is quasiconformal on the interior. For many applications it is important to know that a holomorphic motion can be extended to a holomorphic motion of the entire sphere.
Bers \& Royden \cite{BR} and Sullivan \& Thurston
\cite{ST} proved that there exists a universal $\delta>0$ such
that under the circumstances of the Extension $\lambda$-lemma, the
restriction of $f$ to the parameter disk $\Delta_{\delta}$ of
radius $\delta$ can be extended to a holomorphic motion
$\Delta_{\delta} \times \hat {{\mathbb C}} \mapsto \hat{\mathbb
C}$. S\l{}odkowski 
\cite{Slodkowski} proved the strongest version asserting that $\delta$ is
actually equal to $1$:

\newtheorem*{Slod lemma}{$\lambda$-lemma [S\l{o}dkowski]}
\begin{Slod lemma}
Let  $A\subset \hat{\mathbb C}$. Any holomorphic motion
$f:\Delta\times A \to \hat{\mathbb C}$ extends to a holomorphic
motion $\Delta\times \hat{\mathbb C} \mapsto \hat{\mathbb C}$.
\end{Slod lemma}

S\l{}odkowski's proof builds on the work by Forstneri\v{c} \cite{Forstneric} and \v{S}nirel'man  \cite{Snirelman}.
Astala and Martin \cite{AM} gave an exposition of S\l{}odkowski's proof from the point of view of $1$-dimensional complex analysis. Chirka \cite{ChirkaOka} gave an independent proof using solution to $\bar{\partial}$-equation. (See \cite{Teich} for a detailed exposition of Chirka's proof.) The purpose of this paper is to give a more geometric approach to the proof of the $\lambda$-lemma. 
We take S\l{}odkowski's approach and replace the major technical part in his proof (closedness, see \cite[Theorem 4.1]{AM}) by a geometric pseudoconvexity argument.

The strongest $\lambda$-lemma fails when the dimension of the base
manifold is greater than $1$, even if the base is topologically contractible. This follows from the results of Earl-Kra \cite{EK} and Hubbard \cite{Hu}.

We give the necessary background on holomorphic motions, pseudoconvexity and Hilbert transform in Section \ref{sec:background}. In Section \ref{sec:axiom_of_choice}, we show that the $\lambda$-lemma when $A$ is finite implies the $\lambda$-lemma for arbitrary $A$. We set up the notations and terminology in Section \ref{sec:terminology}. We state the filling theorem for the torus, and explain how it implies the finite $\lambda$-lemma in Section \ref{sec:filling_theorem}. In Section \ref{sec:trapping_disks} we prove H\"{o}lder estimates for disks trapped inside pseudoconvex domains and construct such trapping pseudoconvex domains for ``graphical tori''. We use these estimates to prove the filling theorem in Section \ref{sec:proof}.
\subsection{Acknowledgments} We would like to thank Misha Lyubich, Yakov Eliashberg and the referee for fruitful discussions and useful suggestions.

\section{Background}\label{sec:background}

\subsection{Holomorphic motion}


Let $\Delta$ be a unit disk. Let $A\subset \hat{\mathbb C}$. A {\it holomorphic motion} of  $A$ is a map
$f$: $\Delta\times A \to \hat{\mathbb C}$ such that
\begin{enumerate}
\item for fixed $a\in A$, the map $\lambda\mapsto f(\lambda,a)$ is holomorphic in $\Delta$
\item for fixed $\lambda\in \Delta$, the map $a\mapsto f(\lambda,a)=:f_{\lambda}(a)$ is an injection and
\item the map $f_0$ is the identity on $A$.
\end{enumerate}

\subsection{Pseudoconvexity}\label{sec:pseudo}

Below we give definitions that are sufficient for our purposes.

A $C^2$ smooth function is {\it (strictly) plurisubharmonic} (written (strictly) psh) if its restriction to every complex line is strictly subharmonic. In coordinates 
$z=(z_1,\dots,z_n)$, $u(z)$ is strictly psh if the matrix $\left(\frac{\partial^2 u}{\partial z_j\partial \bar{z}_k}\right)$ is positive definite.

A smoothly bounded domain $\Omega\subset \mathbb C^2$ is {\it strictly pseudoconvex} if there is a smooth, strictly psh
function $\rho$ in a neighborhood of $\bar{\Omega}$ such that $\{\Omega=\rho(z)<0\}.$

\begin{lemma}\label{lem:disk_trapping}Let $\Omega_s\subset \mathbb C^2$ be a
family of pseudoconvex domains with defining functions $\rho_s$,
$s\in[0,1]$. We assume that the family $\rho_s$ is continuous in
$s$. Let $\phi_s:\Delta\mapsto \mathbb C^2$ be a continuous family
of holomorphic non-constant functions that extend continuously to
$\bar{\Delta}$. Set $D_s:=\phi_s(\Delta)$. Suppose $\partial
D_s\subset
\partial \Omega_s$, $s\in [0,1]$. And suppose $D_s\subset
\Omega_s,$ $s\in [0,1)$. Then $D_1\subset \Omega_1$.
\end{lemma}

\begin{proof} Consider the restriction of the functions $\rho_s$ to $D_s$. The functions $\rho_s\circ \phi_s:\Delta\mapsto \mathbb R$ are subharmonic functions, $\rho_1\circ \phi_1$ is the limit of $\rho_s\circ \phi_s$. By the hypothesis of the lemma, $\rho_s\circ \phi_s\leq 0$ on $\Delta$. Therefore,
$\rho_1\circ \phi_1\leq 0$. If the maximum value $0$ is attained
in the interior point, $\rho_1\circ \phi_1\equiv 0$. It implies
that $D_1\subset \partial \Omega_1$, which is impossible.
Therefore, $\rho_1\circ\phi_1<0$ on $\Delta$, and $D_1\subset
\Omega_1$.
\end{proof}

Let $M\subset \mathbb C^2$ be a real two-dimensional
manifold. We say that $p\in M$ is a {\it totally real} point if
$T_pM\cap iT_pM=\{0\}$. $M$ is a totally real manifold if all its
points are totally real. If the manifold $M$ is totally real, it is in fact homeomorphic to the torus (see \cite{Bishop} and \cite{GW}).
Assume $M\subset \partial \Omega$, then one can define a
characteristic field of directions on $M$.

Let $p\in M$. Let $H_p\partial \Omega:=T_p\Omega\cap iT_p\Omega$ be the holomorphic tangent space.
$\langle\xi_p\rangle:=H_p\partial \Omega\cap T_pM$ is called the {\it characteristic
direction}. We denote by $\chi(M, \Omega)$ the characteristic field of directions
(see \cite[Section 16.1]{CE}).

\subsection{Hilbert transform}

A function $u:\mathbb S^1\to \mathbb C$ is {\it H\"{o}lder continuous with exponent} $\alpha$
if there is a constant $A$ such that for all $x,y\in \mathbb S^1$:
$$|u(x)-u(y)|<A |x-y|^{\alpha}.$$

We will consider the space $C^{1,\alpha}(\mathbb S^1)$ of differentiable functions $u$ with $\alpha$- H\"{o}lder continuous derivative. The norm on the space $C^{1, \alpha}(\mathbb S^1)$ is defined by the formula:

$$||u||_{1, \alpha}:=\sup_{x\in \mathbb S^1}|u(x)|+\sup_{x\in \mathbb S^1}|u'(x)|+\sup_{x\neq y\in \mathbb S^1}\frac{|u'(x)-u'(y)|}{|x-y|^{\alpha}}.$$


There exists a unique harmonic extension $u_h$ of the function $u$ to $\Delta$. Let denote by $u_h^*$ the harmonic conjugate of $u_h$, normalized by the condition $u_h(0)=0$. The function $u_h^*$ extends to $\mathbb S^1=\partial \Delta$ as a H\"{o}lder continuous function with exponent
$\alpha$.

For a function $u\in C^{1, \alpha}(\mathbb S^1)$ we define its {\it Hilbert transform} $Hu$ to be the boundary value of the harmonic conjugate function $u_h^*$. By definition, the function $u+iHu$ 
extends as a holomorphic function to the unit disk.

\begin{theorem} The Hilbert transform $H$ is a bounded linear operator on $C^{1,\alpha}(\mathbb S^1)$ and $C^{\alpha}(\mathbb S^1)$.
\end{theorem}
This Theorem makes it convenient for us to work with the spaces $C^{1,\alpha}(\mathbb S^1)$ and $C^{\alpha}(\mathbb S^1)$.

\section{Finite $\lambda$ -lemma}\label{sec:axiom_of_choice}
The first step in the proof of the $\lambda$-lemma is to reduce it
to the $\lambda$-lemma for finitely many points, \cite{MSS}.

\begin{theorem}{\bf The Finite $\lambda$-lemma} Assume $a_1, \dots, a_{n+1}\in \hat{\mathbb C}$, $a_i\neq a_j$ for $i\neq j$. Let
$f:\Delta\times\{a_1, \dots, a_n\}\to \hat{\mathbb C}$ be a holomorphic motion. Then there exists a
holomorphic motion $\tilde{f}: \Delta\times\{a_1,\dots, a_{n+1}\}\to \mathbb C$, so that
$\tilde{f}$ is an extension of $f$.
\end{theorem}

\begin{proof}[Reduction of the $\lambda$-lemma to the finite $\lambda$-lemma (assuming the Extension-$\lambda$ lemma):]
We normalize the holomorphic motion $f$ so that three points $a_1, a_2, a_3$ stay fixed. We can assume $a_1=0$, $a_2=1$, $a_3=\infty$.

Let $\{a_n\}$ be a sequence of points that are dense in $\bar{A}$.
Let $\{z_n\}$ be a sequence of points that are dense in $\mathbb C\backslash \bar{A}$. Let $f_n$ be a holomorphic motion of $a_1,\dots, a_n,z_1,\dots, z_n$, such that
$$f_n(\lambda,a_i)=f(\lambda,a_i).$$
The existence of such holomorphic motion follows from the Finite
$\lambda$-lemma.

For any fixed $z_n$, for $k\geq n$ and $n\geq 3,$ maps $f_k$ are defined at the
point $z_n$, and functions $f_k(*, z_n):\Delta\to \mathbb
C\backslash\{0,1\}$ form a normal family. So we can choose a
convergent subsequence $f_{k}(*,z_n)$. Using the diagonal method, we get
a holomorphic motion $\tilde{f}$, that is well defined for all
$a_n$ and $z_n$ and coincides with $f$ on $a_i$ for all $i$. By the Extension
$\lambda$-lemma, we extend it uniquely to the holomorphic motion
of $\hat{\mathbb{C}}$. By construction,  it coincides with the
holomorphic motion $f$ on the set $A$.
\end{proof}

\section{Notations and Terminology}\label{sec:terminology}

We consider $\mathbb C^2$ with coordinates $(\lambda,w)$. The horizontal direction is parametrized by $\lambda$, the vertical by $w$.
Throughout the paper we consider disks of the form
$$w=g(\lambda)$$
that will depend on two different parameters.
We will use the following notations
$$g:\Delta\times\mathbb S^1\times [0,t_0]\to \mathbb C^2$$
$$g^t_{\xi}(\lambda):=g^t(\lambda,\xi):=g_{\xi}(\lambda,t):=g(\lambda,\xi,t).$$




\subsection{Graphical Torus}

Let $\pi:\mathbb C^2\to \mathbb C$, $\pi(\lambda,w)=\lambda$, be the projection to the first
coordinate.

We say that a torus $\Gamma\subset \{\partial \Delta\}\times \mathbb C$ is a {\bf graphical torus} if
for each $\lambda\in\partial \Delta$, $C_{\lambda}:=\pi^{-1}(\lambda)\in \mathbb C\backslash \{0\}$ is a simple closed curve that
has winding number $1$ around $0$.

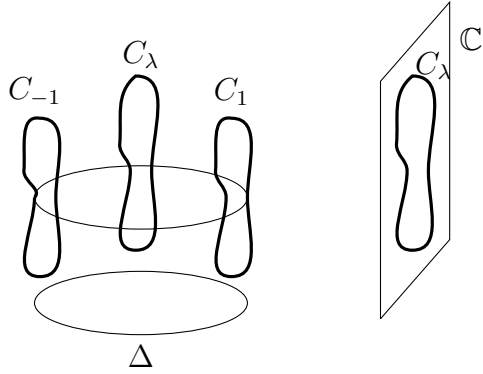
\begin{figure}[h!]
\begin{center}
\begin{tikzpicture}[scale =0.7]
\node at (0,-3) {$\Delta$};

\draw (0,-2) ellipse (2cm and 0.6cm);
\draw (0,0) ellipse (2cm and 0.6cm);

\draw [very thick] (-2,0) to  [in=-90] (-2.2,0.5) to  [out=85, in=170] (-1.9,1.5) to [out=0, in=90] (-1.6,0) to [out=-90, in=0]
(-1.9,-1.5) to [out=180, in=-90] (-2,0);

\node at (-2,2) {$C_{-1}$};

\draw [very thick] (-2+1.8,0+0.5) to  [out=90, in=-90] (-2.2+1.8,0.5+0.5) to  [out=85, in=220] (-1.9+1.8,1.8+0.5) to [out=0, in=90] (-1.6+1.8,0+0.5) to [out=-90, in=0]
(-1.9+1.8,-1.5+0.5) to [out=180, in=-90] (-2+1.8,0+0.5);

\node at (0,2.7) {$C_{\lambda}$};

\draw [very thick] (-2+3.6,0) to  [out=90, in=-90] (-2.2+3.6,0.5) to   [out=85, in=170] (-1.9+3.6,1.5) to [out=0, in=90] (-1.6+3.6,0) to [out=-90, in=0]
(-1.9+3.6,-1.5+0) to [out=180, in=-90] (-2+3.6,0);

\node at (1.7,2) {$C_{1}$};

\draw [very thick] (-2+7,0+0.5) to  [out=90, in=-90] (-2.2+7,0.5+0.5) to  [out=85, in=220] (-1.9+7,1.8+0.5) to [out=0, in=90] (-1.6+7,0+0.5) to [out=-90, in=0]
(-1.9+7,-1.5+0.5) to [out=180, in=-90] (-2+7,0+0.5);

\node at (5.5,2.5) {$C_{\lambda}$};

\draw (4.5,-2.3) to (4.5,2.2) to (5.8,3.7) to (5.8, -0.8) to (4.5,-2.3);

\node at (6.2, 3) {$\mathbb{C}$};

\end{tikzpicture}
\end{center}
\caption{Torus $\Gamma$}
\end{figure}

Thus, the vertical slices $\{C_{\lambda}:\ \lambda\in \mathbb S^1\}$ give a foliation of $\Gamma$. We wish to construct a transverse foliation of $\Gamma$. We will consider holomorphic functions $g_{\xi}:\Delta\to \mathbb C$, which extend continuously to $\bar{\Delta}$ and such that $\gamma_{\xi}:=g_{\xi}(\partial \Delta)\subset \Gamma$. We will construct a family of holomorphic disks such that $\{\gamma_{\xi}: \xi \in \mathbb S^1\}$ form another foliation of the torus $\Gamma$ that is transverse to the original foliation.

\subsection{Family of Graphical Tori}

Let $\{C_{\lambda}^t: t>0, \lambda\in \partial \Delta\}$ be smooth curves, such that
\begin{enumerate}
\item $C^t_{\lambda}$ have winding number $1$ around $0$;
\item for fixed $\lambda$, $C^t_{\lambda}$ form a smooth foliation of $\mathbb C\backslash\{0\}$;
\item there exists $\epsilon>0$, so that $C^t_{\lambda}=\{|w|^2=t\}$ for $t<\epsilon$.
\end{enumerate}

Let
$$\Gamma^t=\{(\lambda,w):\ \lambda\in \partial \Delta, w\in C^t_{\lambda}\}.$$
We set $\Gamma^0=\{(\lambda, 0):\ \lambda\in \partial \Delta\}$. We refer to $\Gamma^t$, $t\geq 0$ as smooth family of graphical tori, though for $t=0$ it degenerates to a circle $\Gamma^0$. The superscript $t$ will be applied to indicate the dependence on the torus $\Gamma^t$.

\subsection{Holomorphic Transverse Foliation of a Graphical Torus}

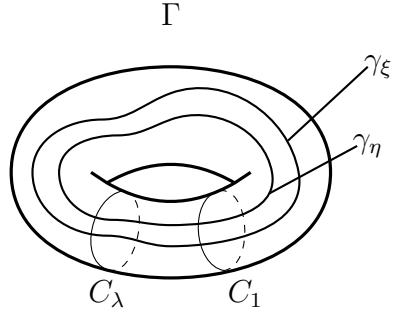
\begin{figure}[h!]
\begin{center}
\begin{tikzpicture}[scale =0.7]
\node at (0,3) {$\Gamma$};

\draw [very thick] (-3,0) to [out=90, in=180] (0,2) to [out=0, in=90] (3,0) to [out=-90, in=0] (0,-2) to [out=180, in=-90] (-3,0);

\draw [very thick] (-1.2,-0.2) to [out=30, in=150] (1.2,-0.2);
\draw [very thick] (-1.5,0) to (-1.2,-0.2) to [out=-30, in=-150] (1.2,-0.2) to (1.5,0);

\draw [thick] (-2.1,0) to [out=90, in=180] (-1,0.7) to [out=0, in=200] (0,1) to [out=20, in=100] (1.9,0) to [out=-80, in=0] (0,-1) to [out=180, in=0] (-1.1,-0.9) to [out=180, in=-90] (-2.1,0);

\draw [thick] (-2.6,0) to [out=90, in=180] (-1.3,1) to [out=0, in=200] (0,1.5) to [out=20, in=100] (2.4,0) to [out=-80, in=0] (0,-1.4) to [out=180, in=0] (-1.2,-1.2) to [out=180, in=-90] (-2.6,0);


\draw  (-0.9, -0.35) to [out=190, in=160] (-1.2,-1.84);
\draw  [dashed] (-0.9, -0.35) to [out=-10, in=-20] (-1.2,-1.84);
\node at (-1.2,-2.3) {$C_\lambda$};

\draw [dashed] (0.9, -0.35) to [out=10, in=10] (1.0,-1.86);
\draw (0.9, -0.35) to [out=195, in=170] (1.0,-1.86);
\node at (1.4,-2.3) {$C_1$};

\draw [thick] (2.2,0.6)--(3.7,2);
\node at (3.9,2) {$\gamma_\xi$};

\draw [thick] (1.85,-0.4)--(3.5,0.5);
\node at (3.7,0.5) {$\gamma_\eta$};

\end{tikzpicture}
\end{center}
\caption{Holomorphic Transverse Foliation of the Torus $\Gamma$}
\end{figure}

Let $\Gamma$ be a graphical torus. Let $g:\Delta\to \mathbb C$ be a holomorphic function that extends continuously to the closure $\bar{\Delta}$.
We say that the function $g:\bar{\Delta}\to \mathbb C$ defines a {\bf holomorphic disk} $D:=\{(\lambda,g(\lambda)):\lambda \in \Delta\}\subset \mathbb C^2$ with a {\bf trace} $\gamma:=\partial D$.

We will construct foliations of graphical tori by traces of holomorphic disks.   To do this, we will require additional properties:

We say that a function $g:\bar{\Delta}\times \mathbb S^1\to\mathbb C$ defines a {\bf holomorphic transverse foliation} of a graphical torus
$\Gamma$ if
\begin{enumerate}
\item $g$ is continuous.  
\item for each $\xi\in \mathbb S^1$, we let $\{\gamma_{\xi}:=g(\lambda,\xi): \lambda\in \partial \Delta\}$. The curves $\gamma_{\xi}$ are simple, pairwise disjoint and define a foliation of $\Gamma$.
\item Let $g_{\xi}(\lambda):=g(\lambda,\xi)$, $g_{\xi}:\Delta\to \mathbb C$ is holomorphic, $g_{\xi}\in C^{1,\alpha}(\bar{\Delta})$
\item $g_{\xi}(\lambda)\neq 0$, for all $\xi\in \mathbb S^1$, $\lambda\in \Delta$
\item $g_{\xi}(\lambda)\neq g_{\eta}(\lambda)$, for every $\lambda\in \Delta$ and distinct $\xi,\eta\in \mathbb S^1$.
\end{enumerate}

We will also consider holomorphic transverse foliations of a smooth family graphical tori $\{\Gamma^t\}$.  This refers to a smooth family of foliations of graphical tori $\Gamma^t$ with the additional assumption that the disks from $\Gamma^{t_1}$ are disjoint from the disks from $\Gamma^{t_2}$ if $t_1\ne t_2$.

In fact the leaves in all of our foliations will be closed, and thus they are also fibrations by curves.

%
%
%

\section{Holomorphic transverse foliations and the Finite $\lambda$-lemma}\label{sec:filling_theorem}

\newtheorem*{fill_th}{Filling Theorem} 

\begin{fill_th} \label{te:torus_foliation} Let $\Gamma$ be a graphical
torus, then there exist a function $g:\bar{\Delta}\times \mathbb
S^1\to \mathbb C$ that defines a holomorphic transverse foliation
of $\Gamma$. Moreover, the foliation is unique in the following strong
sense: if there is a function $h:\bar{\Delta}\to\mathbb C$ that
defines a holomorphic disk with trace in $\Gamma$, and if $h(\lambda)\neq 0$ for
$\lambda\in \Delta$, then there exists $\xi\in \mathbb S^1$ so
that $h=g_{\xi}$.
\end{fill_th}

We need the following slightly stronger statement to deduce the
Finite $\lambda$-lemma:

\newtheorem*{fill_th'}{Filling Theorem$'$}

\begin{fill_th'} Let $\Gamma^t$, $t\in [0,\infty)$ be a family of graphical tori. There exists a function
$g:\bar{\Delta}\times \mathbb S^1\times [0,\infty)\to \mathbb C$ that defines a
holomorphic transverse foliation of the family $\Gamma^t$. And the foliation is unique in the above mentioned strong sense.
\end{fill_th'}

The reduction of the Finite $\lambda$-lemma to Filling Theorem$'$
can be found in \cite{Slodkowski}. 
\begin{proof}[Reduction of the Finite $\lambda$-lemma to Filling Theorem$'$]
Let $f$ be a holomorphic motion of the points $a_1,\dots,a_n$. We
need to extend the motion $f$ to one more point $a_{n+1}$. To
achieve that we construct a holomorphic motion of all of $\mathbb
C$ and pick the leaf that passes through the point $a_{n+1}$.

We normalize the motion so that $a_1=0$, $f(\lambda,0)=0$ for all
$\lambda\in \Delta$. Let $\lambda=re^{i\theta}$. For each $r\in
[0,1)$, $e^{i\theta}\in \mathbb S^1$ the derivative $\frac{\partial
f}{\partial r} (\lambda, a_i)$ defines a vector $v_{\theta}(r,a_i)$ in $\mathbb C$. We
can extend it to a smooth family of vector fields $v_{\theta}(r,\cdot)$
on $\mathbb C$. By integrating the vector field for $r\in [0,1)$ and taking the union of solutions over
$\xi\in \mathbb S^1$,
we get a smooth motion $g:\Delta\times \mathbb C\to \mathbb{C}$ such that
$g(\lambda,a_i)=f(\lambda, a_i)$.

Let $C_0^t$ be a smooth family of simple curves that foliate
$\mathbb C\backslash \{0\}$. We choose the foliation so that
different $a_i$ belong to different curves $C_0^t$. Take $r<1$. Let
$\mathbb S_r=\{\lambda:\,|\lambda|=r\}$. Let
$C^{t}_\lambda=g(\lambda, C^t_0)$ for $\lambda \in \mathbb S_r$.

By Filling Theorem$'$, there exists a holomorphic
motion with the prescribed traces $\Gamma^t_r=\{(\lambda,
C^t_{\lambda}):\,\lambda\in \mathbb S_r\}$. By the uniqueness, it
coincides with $f$ on points $a_1, \dots, a_n$. By taking the
limit as $r\to 1$, we obtain a holomorphic motion of $\mathbb C$
that coincides with $f$ on $a_1,\dots, a_n$.
\end{proof}

\section{Trapping holomorphic disks inside pseudoconvex domains} \label{sec:trapping_disks}

The aim of the section is to prove a priori estimates for the derivative of a disk with the trace in a graphical torus (Corollary \ref{cor:apriori}), which is the heart of our proof of the $\lambda$-lemma. 

\subsection{Estimates for holomorphic disks trapped inside strictly
pseudoconvex domains}

The next theorem is from \cite{BK}, \cite{BG}. We do not use the result of the theorem. We provide the proof to shed light on the technique we use and put the results in
a general context. 

\begin{theorem} \cite{BK}, \cite{BG}\label{te:angle_estimate} Let $\Omega$ be a strictly pseudoconvex domain, and let $M$ be a totally real $2$-dimensional manifold,
$M\subset \partial \Omega$. Let $g:\Delta\to \Omega$ be an injective
holomorphic function that extends as a $C^1$ smooth function to the closure $\bar{\Delta}$. Set
$D=g(\Delta)$. Assume that $\gamma:=\partial D\subset M$. Then there
is a constant $\alpha = \alpha(M, \Omega)$, so that the angle
$\angle(T_p\gamma,\xi_p)>\alpha$ is uniformly large, independently of $D$.
\end{theorem}



\begin{lemma}\label{lem:transv_ch} Under hypothesis of Theorem \ref{te:angle_estimate}, for every point $p\in \gamma$, $T_p\gamma$ is transverse to the characteristic field of directions $\chi(M, \Omega)$.
\end{lemma}

\begin{proof} Let $\rho$ be a strictly psh function such that $\Omega=\{\rho<0\}$. The function $\rho\circ g:\Delta\to \mathbb R$ is subharmonic. Let $p\in \partial \Delta$. By the Hopf Lemma, the radial derivative $\frac{\partial \left(\rho\circ g\right)}{\partial r}(p)>0$. Let $\xi_p$ be a vector that defines the characteristic direction in a point $p$. The normal vector to the disk $g(\Delta)$ in a point $p$ is $iT_p\gamma$. It does not belong to the tangent plane to $\partial \Omega,$ so $iT_p\gamma$ is transverse to $i\xi_p$. Therefore, $T_p\gamma$ is transverse to $\xi_p$.
\end{proof}

Let $n_p$ be the unit outward normal vector to the hypersurface $\partial \Omega$. The vectors $(\xi_p, i\xi_p, n_p, in_p)$ form an
orthonormal basis in $\mathbb C^2\approx \mathbb R^4$ with respect to Euclidean inner product $(\cdot,\cdot)$. The vectors $in_p$ and $\xi_p$ form an orthonormal basis for $T_pM$.   Given $\alpha$, we define a conical neighborhood of $\xi_p$:
$$K_{\alpha}=\{v\in T_pM: (v,\xi_p)>\alpha(v,in_p)\} \subset T_pM.$$

\begin{lemma}\label{lem:family_pseudoconvex} Let $\Omega$ be a strictly pseudoconvex domain, and let $M\subset \partial \Omega$ be totally real.  There
exist $\alpha>0$, and a continuous family of strictly
pseudoconvex domains $\Omega_{\epsilon}$ such that $M\subset \partial
\Omega_{\epsilon}$, and the characteristic fields of directions $\chi(M,
\Omega_{\epsilon})$ fill the cone-fields $K_{\alpha}$.
\end{lemma}

\begin{proof} The manifold $M$ separates $\partial \Omega$ into two parts
$(\partial \Omega)_1$, $(\partial \Omega)_2$. Let $h$ be a smooth
function such that
\begin{enumerate}
\item $h|_{M}=0$;
\item $h|_{(\partial \Omega)_1}>0$, $h_{(\partial \Omega)_2}<0$;
\item $\frac{\partial h}{\partial (i\xi_p)}>0$, for each $p\in M$.
\end{enumerate}
Let us denote by $\vec{n}$ the normal field to the hypersurfaces
$\rho=const$. Since we can identify $T_p\mathbb C^2$ with $\mathbb
C^2$, we can treat the normal vector field $n$ as a function
defined in a neighborhood of $\partial \Omega$. We use the same
letter $n$ for this function. Let $\rho_{\epsilon}(z)=\rho(z+\epsilon h\vec{n})$,
$\Omega_{\epsilon}=\{\rho_{\epsilon}<0\}$. Then there exists $\delta$, so that for
$|\epsilon|<\delta$, $\rho_{\epsilon}$ are plurisubharmonic. Therefore, $\Omega_{\epsilon}$
are strictly pseudoconvex, and characteristic fields of directions
to $\Omega_{\epsilon}$ fill the cone field $K_{\alpha}$.

\end{proof}

\begin{proof}[Proof of Theorem \ref{te:angle_estimate}]

Let $D\subset \Omega$, $\partial D\subset \partial \Omega$. Then
by Lemma \ref{lem:family_pseudoconvex}, there exists a continuous
family of strictly pseudoconvex domains $\Omega_{\epsilon}$, $|\epsilon|<\delta$
so that their characteristic fields of directions fill
$C_{\alpha}$, for some $\alpha>0$. By Lemma
\ref{lem:disk_trapping}, $D\subset \Omega_{\epsilon}$ for $|\epsilon|<\delta$.
Therefore, an angle estimate follows.
\end{proof}

\subsection{Pseudoconvex domains for Graphical Tori}
We wish to obtain the angle estimates for graphical tori. Let $\eta_p$ be a vector that is tangent to the curve
$C_{\lambda}$ in a point $p$. We want to think of $\eta_p$ as a characteristic direction. However, a priori a graphical torus $\Gamma$ does not belong
to a pseudoconvex domain. It belongs to a Levi flat domain $\{|\lambda|=1\}\times \mathbb{C}$. Our strategy is to curve this Levi
flat domain to obtain a family of pseudoconvex domains whose boundaries contain the torus $\Gamma$ and so that characteristic directions span a wedge around $\eta_p$. 

\begin{theorem}\label{te:anlgle_estimate2} Let $\Gamma$ be a graphical torus. Assume that $g:\Delta\to \mathbb C$ defines a holomorphic disk $D$ with the trace $\gamma\subset \Gamma$, $g(\lambda)\neq 0$. Then there exists a constant $\alpha=\alpha(\Gamma)>0$ (independent of $D$) so that the angle $\angle(\eta_p,
T_p\gamma)$ is bounded below by $\alpha$ independently of $D$. 
\end{theorem}

We need Lemmas \ref{lem:phi}, \ref{lem:psi} and \ref{lem:pseudo} to prove Theorem
\ref{te:anlgle_estimate2}.

Consider a family of the graphical tori $\Gamma^t$, $\Gamma^1=\Gamma$. Let $F:\mathbb S^1\times \mathbb C\to \mathbb R$ be a defining function, $F^{-1}(t)=\Gamma^t$. 
Let us extend $F$ to a smooth function $F: \bar{\Delta}\times \mathbb C\to \mathbb R$, so
that $F(\lambda,w)=|w|^2$ for all $\lambda\in \bar{\Delta}$,
$|w|\leq \epsilon$. We can also satisfy the condition $F'_w\neq 0$.

\begin{lemma}\label{lem:phi} There exists a function
$\phi:\bar{\Delta}\times \mathbb C\to \mathbb R_{\geq 0}$, so that $\phi$ is
smooth, $\Delta_{w}\phi>0$, and restriction of $\phi$ to $\mathbb S^1\times \mathbb C$ defines a foliation of $\mathbb S^1\times \mathbb C$ by $\Gamma^t$. We also require that for $|\lambda|=1$
$\phi_{\lambda}^{-1}(1)=C_{\lambda}.$
\end{lemma}

\begin{proof} Let $F(\lambda, w)$ be the extension defined earlier. Let $\rho:\mathbb R_+\to \mathbb R_+$
be an increasing convex function, $\rho(0)=0,$ $\rho(1)=1$. Then $\phi=\rho\circ F$ is also an extension of a defining function of the foliation as well.

\begin{equation}\label{eq:1}
\Delta_w(\rho\circ F)=\frac14\rho''|F_w|^2+\frac14\rho'\Delta_w F
\end{equation}

Since $F'_w(\lambda, w)\neq 0$, when $w\neq 0$, so that $\Delta_w(\rho\circ F)>0$ away from a neighborhood of $w=0$.  In a neighborhood of $0$, $\Delta_w F=4$. By
taking $\rho'(0)>0$, one can insure that $\Delta(\rho\circ F)>0$.

Let us set $\phi=\rho\circ F$, then $\phi_{\lambda}^{-1}=C_{\lambda}$.
\end{proof}

\begin{lemma}\label{lem:psi} There exists a function $\psi:\bar{\Delta}\times \mathbb C\to \mathbb R\cup\{-\infty\}$, so that $\psi$ is smooth, $\Delta_{w} \psi<0$, and restriction of $\psi$ to $\mathbb S^1\times \mathbb C$ defines a foliation of $\mathbb S^1\times \mathbb C$ by $\Gamma^t$. We require that $\psi(\lambda,0)=-\infty$ for all $\lambda\in \bar{\Delta}$. We also require that for $|\lambda|=1$, $\psi_{\lambda}^{-1}(t)=C_{\lambda}$.
\end{lemma}

\begin{proof}
Consider a function $\psi=c \rho\circ \ln F$, where $\rho$ is increasing, concave function, $\rho(-\infty)=-\infty$.

$$\Delta_w(\rho\circ\ln F)=\frac14\rho''\frac{|F_w|^2}{F^2}+\frac14\rho' \Delta_w (\ln F)$$

Since $F'_w\neq 0$ when $w\neq 0$, we can make $\Delta_w(\rho\circ \ln F)<0$. In a neighborhood of $w=0$, $\Delta_w(\ln F)=0$, therefore 
$\Delta_w (\rho\circ \ln F)<0$. By choosing a constant $c$, we can ensure that $\psi_{\lambda}^{-1}(1)=C_{\lambda}$.

\end{proof}

Let $T\Gamma$ be the tangent space of the graphical torus $\Gamma$. Let $K_{\alpha}\subset T\Gamma$ be the cone field:

$$K_{\alpha}:=\{(p, v):\, v\in T_pT, (v,\eta_p)>\alpha(v, \frac{\partial }{\partial \theta})\}.$$
$$K^{\circ}_{\alpha}:=\{(p,v)\in K_{\alpha}:\, v\neq c \eta_p, c\in \mathbb R \}$$

\begin{lemma} \label{lem:pseudo} For a graphical torus $\Gamma$, there exist a family of pseudoconvex domains $\Omega_{\epsilon}$, $\epsilon\in [-\delta, 0)\cup(0,\delta]$ and $\alpha>0$, so that $\Gamma\subset \partial \Omega_{\epsilon}$ and characteristic directions $\chi(T,\Omega_{\epsilon})$ fill $K^{\circ}_{\alpha}$.
\end{lemma}

\begin{proof} Take
$$\omega_{\epsilon}:=\frac{1}{\epsilon}(|\lambda|^2-1)+\phi,$$
where $\phi$ is a function constructed in Lemma \ref{lem:phi}.

$$\mbox{Hess}\,\omega_{\epsilon}=\left( \begin{array}{ll}\frac{1}{\epsilon}+\frac{\partial^2\phi}{\partial \lambda\partial
\overline{\lambda}} & \frac{\partial^2\phi}{\partial w\partial\overline{\lambda}}\\ \frac{\partial^2\phi}{\partial\overline{w}\partial\lambda} & \Delta_w\phi\end{array}\right)$$

For small enough $\epsilon$, the Hessian is positive definite, so the function $\omega_{\epsilon}$ is strictly plurisubharmonic.
The domains
$$\Omega_{\epsilon}=\{(\lambda, w):\,\omega_{\epsilon}(\lambda, w)< 1\}.$$
are strictly pseudoconvex for small $\epsilon$.

Let $D$ be a holomorphic disk with the trace in $\Gamma$.
The domains $\Omega_{\epsilon}$ converge to
$|\lambda|<1$.  Therefore, by Lemma~\ref{lem:disk_trapping}, the disk $D$ is trapped in
$\Omega_{\epsilon}$ for all small enough $\epsilon$.

For small $\epsilon$, the function
$$\sigma_{\epsilon}(\lambda,w):=\frac{1}{\epsilon}(|\lambda|^2-1)-\psi$$
is strictly plurisubharmonic.  By the same reasoning, the disks are trapped in
$$\Sigma_{\epsilon}=\{(\lambda, w):\, \sigma_{\epsilon}<-1\}$$
when $\epsilon$ is sufficiently small.
\end{proof}

\begin{proof}[Proof of Theorem \ref{te:anlgle_estimate2}]
By Lemma \ref{lem:transv_ch}, the tangent $T_p\gamma$ is transverse to characteristic directions. Therefore, the angle estimate follows.
\end{proof}

\begin{corollary} \label{cor:apriori} Let $g:\Delta\to \mathbb C$ define a holomorphic disk with the trace in $\Gamma$, $g(\lambda)\neq 0$ for $\lambda\in \Delta$. Assume that  $g\in C^1(\bar{\Delta}).$  Then there exists $C$ depending only on $\Gamma$ such
that $|g'(\lambda)|<C$ for all $\lambda\in \bar{\Delta}$. The derivative estimate stays valid for graphical tori that are
small perturbations of $\Gamma$.
\end{corollary}

\begin{proof} It is enough to estimate $g'(\lambda)$ for $|\lambda|=1$. Then $\lambda=e^{i\theta}$,
so $|g'_{\lambda}|=|g'_{\theta}|$. Let $u,v,\theta,r$ be an
orthornormal system of coordinates in a neighborhood of $\Gamma$. We
assume that $u|_{\Gamma}$ is a coordinate along $C_{\lambda}$ and $v,u$
are coordinates in $\lambda=const$ plane. Then
$g'_{\theta}=u'_{\theta}$ and the angle estimate implies that
$|u'_{\theta}|$ is uniformly bounded from below.
\end{proof}



\section{Proof of the Filling Theorem}\label{sec:proof}

The proof is by continuity method. At many points we follow the treatment of \cite{AM}. For each $\lambda\in \mathbb
S^1$ we can foliate interior of $C^t_{\lambda}\backslash
\{0\}$ by simple smooth curves $C_{\lambda}^s$, $s\in (0,t)$ so that
\begin{enumerate}
\item $C^s_{\lambda}=\{|z|=s\}$ for $s\leq \epsilon$;
\item $C^s_{\lambda}$ depend smoothly on $\lambda$.
\end{enumerate}

Let
$$\Gamma^t=\{(\lambda,w):\, \lambda\in \mathbb S^1, w\in
C^t_{\lambda}\}$$
$$\Gamma^0=\{(\lambda, 0):\, \lambda\in \mathbb S^1\}$$

$\Gamma^t$ by definition is a smooth family of graphical tori, ${\Gamma}^1=\Gamma$.
For $t\leq \epsilon$, the tori $\Gamma^t$ are foliated by the vertical
leaves $w=\mbox{const}$. We will prove that the set $S$ of
parameters $t$ such that $\Gamma^t$ is foliated is open and closed in
$[0,1]$, so $S=[0,1]$, and the torus $\Gamma$ is foliated. Moreover, we
will prove that the foliation is unique in the strong sense.

Let $F:\mathbb S^1\times \mathbb C\to \mathbb R$ be a
defining function of the foliations $C^t_{\lambda}$. For each
fixed $\lambda$,
$$C^t_{\lambda}=\{(\lambda, w): F(\lambda, w)=t\}.$$
The function $F$ depends smoothly on $\lambda$. We assume that
$F'_{w}(\lambda, w)\neq 0$ for $w\neq 0$, $\lambda\in \mathbb
S^1.$

\begin{lemma} \label{lem:winding_number} Assume that the winding
number of a curve $\{\gamma(\lambda):\lambda\in \mathbb
S^1_{\lambda}\}$ around $0$ is equal to zero. Then the winding
number of the curve $\{F'_w(\lambda,
\gamma(\lambda)):\,\lambda\in\mathbb S^1\}$ around $0$
is equal to zero.
\end{lemma}
\begin{proof} There is a homotopy of the curve $\gamma,$ $G:\gamma\times [0,1]\to \mathbb C\backslash\{0\}$ so that $G(\gamma\times\{0\})=\gamma$, $G(\gamma\times\{1\})=\mbox{const}$. The winding
number of the curves $\{F'_w(\lambda, \gamma^t(\lambda)):\, \lambda
\in \mathbb S^1\}$ around $0$ is well defined, so it stays constant. Hence, it
is equal to zero.
\end{proof}

\subsection{Regularity} 

\begin{theorem}\label{te:regularity} Let $\Gamma$ be a graphical torus. Let $g:\bar{\Delta}\to \mathbb C$ be a function that defines a holomorphic disk with the trace $g(\partial \Delta) \in \Gamma.$ Assume $g'\in L^{\infty}(\Delta)$, $g\neq 0$ $\forall \lambda\in \Delta$. Then $g\in C^{1,\alpha}(\bar{\Delta})$, $0<\alpha<1$.
\end{theorem}
\begin{proof} We include $\Gamma$ into a family of graphical tori $\Gamma^t$ with $\Gamma^1=\Gamma$. Let $F:\mathbb S^1\times \mathbb C\to \mathbb R$ be a defining function for $\Gamma^t$, $F^{-1}(t)=\Gamma^t$. 
Since the trace of $g$ is in $\Gamma$ we have equation:

\begin{equation}\label{eq:def}
F(\lambda,g(\lambda))=1. 
\end{equation}

Let $\lambda=e^{i\theta}.$ Since $g'\in L^{\infty}(\Delta)$, the bounded radial limits exist almost everywhere. The function $g$ extends to be $C^{\alpha}$ on the closed disk, and the partial derivative $g_{\theta}$ exist a.e. We differentiate equation (\ref{eq:def}) a.e. with respect to $\theta$ and obtain: 
\begin{equation}\label{eq:der}\lambda i F_{\lambda}(\lambda,g(\lambda))-\mbox{Im}\left(F_w(\lambda,g(\lambda))g'(\lambda)\lambda\right)=0.\end{equation}

The winding number of $\{g(\lambda):\lambda\in \mathbb S^1\}$ around $0$ is zero, and by Lemma \ref{lem:winding_number}, the winding number of $\{F_w(\lambda,g(\lambda)):\,\lambda \in\mathbb S^1\}$ around $0$ is zero as well. Thus we can take the logarithm and obtain 
$$F_w(\lambda, g(\lambda))=e^{a(\lambda)+ib(\lambda)}.$$
The left hand-side is $\alpha$-H\"{o}lder continuous, so $b(\lambda)$ is $\alpha$-H\"{o}lder continuous function, and so is its Hilbert transform $Hb(\lambda)$. Thus equation (\ref{eq:der}) becomes
$$\mbox{Im} \left(\lambda e^{Hb(\lambda)-ib(\lambda)} g'(\lambda)\right)=e^{-a(\lambda)}F'_{\lambda}(\lambda, g(\lambda))\lambda$$ for almost every $\theta$. Since the right hand side is $C^{\alpha}$, so is the
left hand side. Further the left hand side is the imaginary part of an analytic function 
so the function  
$\lambda e^{Hb(\lambda)-ib(\lambda)} g'(\lambda)$ itself is $C^{\alpha}$. Therefore, $g'\in C^{\alpha}(\bar{\Delta})$.
\end{proof}

\subsection{Openness}

In \cite{B}, the stability of foliation by holomorphic disks is proved if one starts from the standard torus.

\begin{theorem}\label{te:open} Let $\Gamma^t$ be a family of graphical tori, $t\in
[0,\infty)$. Assume that a function $g^{t_0}:\bar{\Delta}\times \mathbb
S^1\to \mathbb C$ defines a holomorphic transverse foliation of a
graphical torus $\Gamma^{t_0}$. Then there exists $\delta$ and
a function $\tilde{g}:\bar{\Delta}\times \mathbb S^1\times
(t_0-\delta,t_0+\delta)\to \mathbb C$ that defines a transverse
holomorphic foliation of $\Gamma^t$ for $|t-t_0|<\delta$.
\end{theorem}

\begin{proof}
Hilbert transform
$$H:C^{1,\alpha}(\mathbb S^1)\to C^{1,\alpha}(\mathbb S^1)$$
is a bounded linear operator. We change the standard normalization $Hu(0)=0$ to $Hu(1)=0$. We denote by $C_{\mathbb R}^{1,\alpha}(\mathbb S^1)\subset C^{1,\alpha}(\mathbb S^1)$ be the subspace of real-valued functions. The curve
$\{g^{t_0}_{\xi}(\lambda), \lambda\in \mathbb S^1\}$
has winding number $0$ around zero, since
$g^{t_0}_{\xi}(\lambda)\neq 0$ for $\lambda\in \Delta$. Therefore, by Lemma \ref{lem:winding_number}, the curve
$\{F_w(\lambda,g^{t_0}_{\xi}(\lambda)):\, \lambda \in \mathbb S^1\}$ has winding number $0$
around $0$:
$$F_w(\lambda,g_{\xi}^{t_0}(\lambda))=e^{a_{\xi}(\lambda)+ib_{\xi}(\lambda)},$$
where $\alpha_{\xi}(\lambda)$, $b_{\xi}(\lambda)$ are H\"{o}lder continuous
with exponent $\alpha$. Thus, $Hb_{\xi}(\lambda)$ is H\"{o}lder
continuous as well.

$X_{\xi}(\lambda):=e^{Hb_{\xi}(\lambda)-ib_{\xi}(\lambda)}$ is a holomorphic
function on $\Delta$ and is proportional to the normal vector to
$C^t_{\lambda}$ in points $(\lambda,g^t_{\xi}(\lambda))$.

Functions of the form $(u(\lambda)+iHu(\lambda))X_{\xi}(\lambda)$ give
all holomorphic functions that are H\"{o}lder continuous up to the
boundary with the condition that $(u_{\xi}(1)+i Hu_{\xi}(1))X_{\xi}(1)$ is proportional
to the normal vector to $C_1^t$ in a point $g^t_{\xi}(1)$. There
exists an $\epsilon$ such that for each point $\eta\in C^t_1$,
$|t-t_0|<\epsilon$, there is only one normal vector that
intersects $C_1^t$ in a point $\eta$.

The space $C^0(\mathbb S^1, C^{1,\alpha}(\mathbb
S^1))$ is a Banach space with the norm
$$||u||=\sup_{\xi\in \mathbb S^1, \lambda\in \mathbb S^1} |u_{\xi}(\lambda)|+\sup_{\xi\in \mathbb S^1, \lambda\in \mathbb S^1} |u'_{\xi}(\lambda)|+\sup_{\xi\in \mathbb S^1,\lambda_1\neq \lambda_2\in \mathbb S^1}\frac{|u'_\xi(\lambda_1)-u'_{\xi}(\lambda_2)|}{|\lambda_1-\lambda_2|^{\alpha}}.$$
Consider an operator 
$${\cal F}:\mathbb R_t\times C^0(\mathbb S^1_{\xi},
C_{\mathbb R}^{1,\alpha}(\mathbb S^1_{\lambda}))\to C^0(\mathbb S_{\xi}^1,
C_{\mathbb R}^{1,\alpha}(\mathbb S_{\lambda}^1)):$$
where ${\cal F}$ is a function of two variable $(t,u_{\xi})$. We consider function $u_\xi(\lambda)$ as an element of $C^0({\Bbb S}_\xi^1,C^{1,\alpha}({\Bbb S}_\lambda^1)).$ 
$${\cal F}(t,u): {\Bbb S}^1\ni (t,\xi)\to F \left(  \lambda, g_\xi(\lambda) + (u_\xi + i H u_\xi)X_\xi(\lambda) \right)   - t  \in C^{1,\alpha}({\Bbb S}_\lambda^1) $$
For $0<\alpha<1$, $H$ is a bounded linear operator, so ${\cal F}$ is a continuous mapping of Banach spaces.  Further, when ${\cal F}$ is considered as a 
map from ${\Bbb R}\times {\Bbb S}^1_\xi$ to $C^{0,\alpha}({\Bbb S}^1_\lambda)$, it is  differentiable, and we compute the differential of ${\cal F}$ at $u_\xi=0$ in the direction $\delta u_\xi$: 
$$D{\cal F}(t,0; \delta u_\xi) = e^{a_\xi(\lambda) - Hb_\xi(\lambda)}\delta u_\xi(\lambda).$$
Since  ${\cal F}(t,0; \delta u_\xi)$ is an invertible linear operator, we can define $u^t_{\xi}$ as the unique element of $C^0({\Bbb S}_\xi^1,C^{1,\alpha}({\Bbb S}_\lambda^1))$ satisfying
${\cal F}(t, u_{\xi}^t)=0.$ 
And the function $\tilde{g}(\lambda, \xi,t)=g^{t_0}_{\xi}(\lambda)+u_{\xi}^t(\lambda)$ defines a
holomorphic transverse foliation and is of class $C^{1,\alpha}$ on
$\bar{\Delta}$. By continuity, for $\xi\neq \eta$, $g^t_{\xi}(\lambda)\neq g^{t}_{\eta}(\lambda)$ for $\lambda\in \Delta$.

\end{proof}
This also gives us the openness for one disk.

\begin{theorem}\label{te:open_leaf} Let $\Gamma^t$, $t\in I$ be a family of graphical tori. Let $g^{t_0}:\bar{\Delta}\to \mathbb C$ be a function that defines a holomorphic disk with the trace in the torus $\Gamma^{t_0}$. Assume that $g^{t_0}\in C^{1,\alpha}(\bar{\Delta})$, 
$g^{t_0}(\lambda)\neq 0$ for $\lambda\in \Delta$. 
Then there exists $\delta$ and a continuous function $g:\bar{\Delta}\times(t_0-\delta,t_0+\delta)\to \mathbb C$ such that $g^t(\lambda):=g(\lambda,t)$ defines a holomorphic disk with the trace in $\Gamma^t$ and $g^t\in C^{1,\alpha}(\bar{\Delta})$, $g^t(\lambda)\neq 0$ for $\lambda\in \Delta$.
\end{theorem}

\subsection{Closedness}

\begin{theorem} Let $\Gamma^t$, $t\in [0, \infty)$ be a family of graphical tori.
Suppose that there exists $g:\bar{\Delta}\times \mathbb S^1\times
[0,t_0)\to \mathbb C$ that defines a holomorphic transverse
foliation of $\Gamma^t$. Then $g$ can be extended to
$g:\bar{\Delta}\times \mathbb S^1\times [0,t_0]\to \mathbb C$ that
defines a holomorphic transverse foliation.
\end{theorem}

\begin{proof} By Corollary \ref{cor:apriori}, there exists $C$ that depends only on $\Gamma^{t_0}$ so that $|\left(g_{\xi}^t\right)'|<C$ for $t<t_0$ close to $t_0$. Since the space of bounded holomorphic
functions on $\Delta$ is compact, we can pass to the limit. Let $g_{\xi}^{t_0}$ be the limits, $|\left(g^{t_0}_{\xi}\right)'|\leq C$. By Regularity Theorem \ref{te:regularity}, 
$g_{\xi}^{t_0}\in C^{1,\alpha}(\bar{\Delta})$.
\end{proof}

This also give us closedness for a family of disks.

\begin{theorem}\label{te:close_leaf} Let $\Gamma^t$, $t\in [0, \infty)$ be a family of graphical tori.
Assume that $g:\bar{\Delta}\times [0,t_0)\to \mathbb C$ is a
continuous function such that $g^t(\lambda):=g(\lambda,t)$ defines
a holomorphic disk with the trace in $\Gamma^t$, $g^t\in C^{1,\alpha}(\bar{\Delta})$, $g^t(\lambda)\neq 0$ for $\lambda\in \Delta$. Then $g$ can be extended
to a continuous $g:\bar{\Delta}\times [0,t_0]\to \mathbb C$ such
that $g^{t_0}$ defines a holomorphic disk with the trace in $\Gamma^{t_0}$, $g^{t_0}(\lambda)\neq 0$ for $\lambda\in \Delta$.
\end{theorem}



\subsection{Uniqueness} 

Let $\Gamma^c=\{(\lambda, w):\, |w|=c, |\lambda|=1\}$ be standard tori.
Let $g:\Delta\to \mathbb C$ be a function that defines a holomorphic disk with the trace of $g(\partial \Delta)\in \Gamma^c$. By the Minimum Modulus 
Theorem, $\min$ of $g$ is attained on the boundary.
Maximum modulus is attained on the boundary as well. So
$|g(\lambda)|=\mbox{const}$. Therefore, $g(\lambda)=\mbox{const}$.

\begin{theorem}\label{te:loc_uniqueness} Let $\Gamma$ be a graphical torus. Let $g,h:\bar{\Delta}\to \mathbb
C$ be functions that define holomorphic disks with traces in $\Gamma$, $g(\lambda)\neq 0$, $h(\lambda)\neq 0$ for $\lambda \in \Delta$.
Assume that $g(1)=h(1)$. Then there exists $\epsilon$ such that $|g(\lambda)-h(\lambda)|<\epsilon\, \mbox{for}\, \lambda\in \mathbb
S^1$ implies $g(\lambda)\equiv h(\lambda)$.

Note that the same $\epsilon$ works for tori close to $\Gamma$.
\end{theorem}

\begin{proof}

Let $a(\lambda,s)=F(\lambda,g(\lambda)+s(h(\lambda)-g(\lambda)))$,
$a(\lambda,0)=a(\lambda, 1)=t$. Then
\begin{equation}\label{eq:nrh}
\int_0^1 a_s ds=0=Re(h(\lambda)-g(\lambda))\int_0^1 F_w
\big(g(\lambda)+s(h(\lambda)-g(\lambda))\big)ds.
\end{equation}

The winding number of the curve $\{g(\lambda):\lambda\in \mathbb
S^1\}$ around $0$ is equal to zero. Hence, by Lemma
\ref{lem:winding_number}, the winding number of 
$$
\{F_w(\lambda, g(\lambda)):\,\lambda \in \mathbb S^1\}
$$
around $0$ is
equal to zero.

Therefore, for small enough $\epsilon$, the winding number of the curve
$$\left\{\int_0^1 F_w\left(g(\lambda)+s(h(\lambda)-g(\lambda)\right)ds:\,\lambda \in \mathbb S^1 \right\}$$ around $0$ is equal to zero, so

$$
\int_0^1 F_w\big(g(\lambda)+s(h(\lambda)-g(\lambda))\big)ds=e^{a(\lambda)+ib(\lambda)}.$$

The function $b(\lambda)$ is a bounded H\"{o}lder continuous function.

By equation (\ref{eq:nrh}), $\arg \left(g(\lambda)-h(\lambda)\right)=\frac{\pi}{2}-b(\lambda)$, where $b(\lambda)$ is a bounded function. It contradicts the fact that $g(1)-h(1)=0$. 
\end{proof}

\subsection{Global Uniqueness}
\begin{theorem}\label{te:uniqueness} Let $\Gamma$ be a graphical torus. Let $g^1,h^1:\bar{\Delta}\to \mathbb
C$ be functions that define holomorphic disks with traces in $\Gamma$, $g^1, h^1\in C^{1,\alpha}(\bar{\Delta})$.
Assume that $g^1(1)=h^1(1)$. Then $g^1(\lambda)=h^1(\lambda)$.
\end{theorem}

\begin{proof}
We include torus $\Gamma$ into a family of
graphical tori $\Gamma^t$, $t\in [0,1]$, $\Gamma^1=\Gamma$. By Theorems
\ref{te:open_leaf}, \ref{te:close_leaf}, there exist functions
$g,h:\bar{\Delta}\times[0,1]\to \mathbb C$ such that $g(\lambda, 1)=g^1$, $h(\lambda,1)=h^1$ and
$g^t(\lambda):=g(\lambda, t)$ define holomorphic disks with the traces in tori $\Gamma^t$, . There
exists $\epsilon$ such that for $t<\epsilon$, $\Gamma^t=\{(\lambda, w):\, |w|=t, \lambda \in \mathbb S^1_{\lambda}\}$, $t\in [0,1]$ are standard tori with uniqueness of solutions. For $t<\epsilon$, $g^t\equiv h^t$. Let $t_0=\sup \{t:\, h^t\equiv g^t\}$. If $t_0\neq 1$, then by applying Theorem \ref{te:loc_uniqueness}, we get a contradiction.
\end{proof}

At this point we have proved the Filling theorem. For Filling Theorem$'$ the only statement remains is to show that disks for $\Gamma^{t_1}$ are disjoint from disks
for $\Gamma^{t_2}$ when $t_1\neq t_2$. Suppose $D^{t_j}$ is a disk with boundary in $\Gamma^{t_j}$. If $D^{t_1}\cap D^{t_2}\neq \emptyset$, then since the traces are in 
$\Gamma^{t_1}$ and $\Gamma^{t_2}$ we will have 
$D_{\xi_1}^{t_1}\cap D_{\xi_2}^{t_1}\neq \emptyset$ for all $\xi_1,\xi_2\in \mathbb S^1$. By Filling Theorem, the disks $D_{\xi_1}^{t_1}\cap D_{\xi_2}^{t_1}=\emptyset$ for $\xi_1\neq\xi_2$. However, there is a continuous family of disks $D_{\xi}^t$, $t\in [t_1,t_2]$, which is a contradiction.

\bibliographystyle{plain}
\bibliography{ref_llemma}

\begin{bibdiv}
\begin{biblist}

\bib{AM}{article}{
      author={Astala, K.},
      author={Martin, G.J.},
       title={Holomorphic motions},
        date={2001},
     journal={Report. Univ. Jyv\"{a}skyl\"{a}},
      volume={83},
       pages={27\ndash 40},
}

\bib{B}{article}{
      author={Bedford, E.},
       title={Stability of the polynomial hull of $\mathbb{T}^2$},
        date={1981},
     journal={Annali della Scuola Normale Superiore di Pisa, Classe di
  Scienze},
      volume={8},
      number={2},
       pages={311\ndash 315},
}

\bib{BG}{article}{
      author={Bedford, E.},
      author={Gaveau, B.},
       title={Envelopes of holomorphy of certain $2$-spheres in
  $\mathbb{C}^2$},
        date={1983},
     journal={Amer. J. Math.},
      volume={105},
       pages={975\ndash 1009},
}

\bib{BK}{article}{
      author={Bedford, E.},
      author={Klingenberg, W.},
       title={On the envelope of holomorphy of a $2$-sphere in $\mathbb{C}^2$},
        date={1991},
     journal={J. Amer. Math. Soc.},
      volume={3},
       pages={623\ndash 646},
}

\bib{BR}{article}{
      author={Bers, L.},
      author={Royden, H.},
       title={Holomorphic families of injections},
        date={1986},
     journal={Acta Math.},
      volume={157},
       pages={259\ndash 286},
}

\bib{Bishop}{article}{
      author={Bishop, E.},
       title={Differentiable manifolds in complex {Euclidean} space},
        date={1965},
     journal={Duke Math. J.},
      volume={32},
       pages={1\ndash 21},
}

\bib{ChirkaOka}{article}{
      author={Chirka, E.M.},
       title={On the propagation of holomorphic motions},
        date={2004},
     journal={Dokl. Akad. Nauk},
      volume={397(1)},
       pages={37\ndash 40},
}

\bib{CE}{book}{
      author={Cielebak, K.},
      author={Eliashberg, Y.},
       title={From {Stein} to {Weistein} and back. {Symplectic} geometry of
  affine complex manifolds},
      series={American Mathematical Society Colloquium Publications},
   publisher={Am. Math. Soc.},
        date={2012},
      volume={59},
}

\bib{EK}{article}{
      author={Earle, C.J.},
      author={Kra, I.},
       title={On holomorphic mappings between {Teichm\"{u}ller} spaces},
        date={1974},
     journal={Contributions to Analysis, Academic Press (New York)},
       pages={107\ndash 124},
}

\bib{Forstneric}{article}{
      author={Forstneri\v{c}, T.},
       title={Polynomial hulls of sets that fiber over the circle},
        date={1988},
     journal={Indiana Univ. Math. J},
      volume={37},
       pages={869\ndash 889},
}

\bib{GW}{article}{
      author={Greenfield, S.},
      author={Wallach, N.},
       title={Extendibility properties of submanifolds of $\mathbb{C}^2$},
        date={1970},
     journal={Proc. Carolina Conf. on Holomorphic Mappings and Minimal Surfaces
  (Chapel Hill, N.C.)},
       pages={77\ndash 85},
}

\bib{Hu}{article}{
      author={Hubbard, J.~H.},
       title={Sur les sections analytiques de la courbe universelle de
  {Teichm\"{u}ller}},
        date={1976},
     journal={Mem. Am. Math. Soc.},
      volume={4},
}

\bib{Teich}{book}{
      author={Hubbard, J.~H.},
       title={Teichm\"{u}ller theory and applications to geometry, topology and
  dynamics},
   publisher={Matrix Editions},
        date={2006},
      volume={1},
}

\bib{L}{article}{
      author={Lyubich, M.},
       title={Some typical properties of the dynamics of rational maps},
        date={1983},
     journal={Russian Math. Surveys},
      volume={38},
       pages={154\ndash 155},
}

\bib{MSS}{article}{
      author={Ma\~{n}\'{e}, R.},
      author={Sad, P.},
      author={Sullivan, D.},
       title={On the dynamics of rational maps},
        date={1983},
     journal={Ann. Sci. \'{E}cole Norm. Sup.},
      volume={16},
       pages={193\ndash 217},
}

\bib{Slodkowski}{article}{
      author={S\l{}odkowski, Z.},
       title={Holomorphic motions and polynomial hulls},
        date={1991},
     journal={Proc. Amer. Math Society},
      volume={111},
       pages={347\ndash 355},
}

\bib{ST}{article}{
      author={Sullivan, D.},
      author={Thurston, W.},
       title={Extending holomorphic motions},
        date={1986},
     journal={Acta Math},
      volume={157},
       pages={243\ndash 257},
}

\bib{Snirelman}{article}{
      author={\v{S}nirel’man, A.I.},
       title={The degree of a quasiruled mapping and the nonlinear {Hilbert}
  problem},
        date={1972},
     journal={Mat.Sb. (N.S.)},
      volume={89(131)},
       pages={366\ndash 389},
}

\end{biblist}
\end{bibdiv}
\end{document}